\documentclass[a4paper,10pt, reqno]{amsart}
\usepackage{amssymb,latexsym,amsmath,epsfig,amsthm} %% Add other packages as necessary 

\theoremstyle{plain}
\newtheorem{theorem}{Theorem}
\newtheorem*{theorem*}{Theorem}

\newtheorem*{corollary*}{Corollary}
\newtheorem{lemma}{Lemma}
\newtheorem*{lemma*}{Lemma}

\newtheorem*{proposition*}{Proposition}

\newtheorem*{conjecture*}{Conjecture}
\theoremstyle{definition}

\newtheorem*{definition*}{Definition}
\theoremstyle{remark}
\newtheorem{remark}{Remark}
\newtheorem*{remark*}{Remark}
\theoremstyle{example}
\newtheorem{example}{Example}
\newtheorem*{example*}{Example}

\begin{document}
\title[Cantor series and functions]{Certain functions defined in terms of Cantor series}
\author{Symon Serbenyuk}
\address{ 
  45~Shchukina St. \\
  Vinnytsia \\
 21012 \\
  Ukraine}
\email{simon6@ukr.net}


% Key words

\begin{abstract}

The present article is devoted to certain examples of functions whose argument represented in terms of Cantor series. 

\end{abstract}
\maketitle

%% \thispagestyle{empty}

%%%%%%%%%%%%%%%%%%%%%%%%%%%%%%%%%%%%%%%%%%%%%%%%%%%%%%%%%%%%%%%%%%%%%%%%

\section{Introduction}

Let $Q\equiv (q_k)$ be a fixed sequence of positive integers, $q_k>1$,  $\Theta_k$ be a sequence of the sets $\Theta_k\equiv\{0,1,\dots ,q_k-1\}$, and $\varepsilon_k\in\Theta_k$.

The Cantor series expansion 
\begin{equation}
\label{eq: Cantor series}
\frac{\varepsilon_1}{q_1}+\frac{\varepsilon_2}{q_1q_2}+\dots +\frac{\varepsilon_k}{q_1q_2\dots q_k}+\dots
\end{equation}
of $x\in [0,1]$,   first studied by G. Cantor in \cite{C1869}. It is easy to see that the Cantor series expansion is the  $q$-ary expansion
$$
\frac{\alpha_1}{q}+\frac{\alpha_2}{q^2}+\dots+\frac{\alpha_n}{q^n}+\dots
$$
of numbers  from the closed interval $[0,1]$ whenever the condition $q_k=q$ holds for all positive integers $k$. Here $q$ is a fixed positive integer, $q>1$, and $\alpha_n\in\{0,1,\dots , q-1\}$.

By $x=\Delta^Q _{\varepsilon_1\varepsilon_2\ldots\varepsilon_k\ldots}$  denote a number $x\in [0,1]$ represented by series \eqref{eq: Cantor series}. This notation is called \emph{the representation of $x$ by Cantor series \eqref{eq: Cantor series}.}

We note that certain numbers from $[0,1]$ have two different representations by Cantor series \eqref{eq: Cantor series}, i.e., 
$$
\Delta^Q _{\varepsilon_1\varepsilon_2\ldots\varepsilon_{m-1}\varepsilon_m000\ldots}=\Delta^Q _{\varepsilon_1\varepsilon_2\ldots\varepsilon_{m-1}[\varepsilon_m-1][q_{m+1}-1][q_{m+2}-1]\ldots}=\sum^{m} _{i=1}{\frac{\varepsilon_i}{q_1q_2\dots q_i}}.
$$
Such numbers are called \emph{$Q$-rational}. The other numbers in $[0,1]$ are called \emph{$Q$-irrational}.

Let $c_1,c_2,\dots, c_m$ be an
ordered tuple of integers such that $c_i\in\{0,1,\dots, q_i-~1\}$ for $i=\overline{1,m}$. 

\emph{A cylinder $\Delta^Q _{c_1c_2...c_m}$ of rank $m$ with base $c_1c_2\ldots c_m$} is a set of the form
$$
\Delta^Q _{c_1c_2...c_m}\equiv\{x: x=\Delta^Q _{c_1c_2...c_m\varepsilon_{m+1}\varepsilon_{m+2}\ldots\varepsilon_{m+k}\ldots}\}.
$$
That is any cylinder $\Delta^Q _{c_1c_2...c_m}$ is a closed interval of the form
$$
\left[\Delta^Q _{c_1c_2...c_m000}, \Delta^Q _{c_1c_2...c_m[q_{m+1}][q_{m+2}][q_{m+3}]...}\right].
$$

Define \emph{the shift operator $\sigma$ of expansion \eqref{eq: Cantor series}} by the rule
$$
\sigma(x)=\sigma\left(\Delta^Q _{\varepsilon_1\varepsilon_2\ldots\varepsilon_k\ldots}\right)=\sum^{\infty} _{k=2}{\frac{\varepsilon_k}{q_2q_3\dots q_k}}=q_1\Delta^{Q} _{0\varepsilon_2\ldots\varepsilon_k\ldots}.
$$

It is easy to see that 
\begin{equation*}
\label{eq: Cantor series 2}
\begin{split}
\sigma^n(x) &=\sigma^n\left(\Delta^Q _{\varepsilon_1\varepsilon_2\ldots\varepsilon_k\ldots}\right)\\
& =\sum^{\infty} _{k=n+1}{\frac{\varepsilon_k}{q_{n+1}q_{n+2}\dots q_k}}=q_1\dots q_n\Delta^{Q} _{\underbrace{0\ldots 0}_{n}\varepsilon_{n+1}\varepsilon_{n+2}\ldots}.
\end{split}
\end{equation*}

Therefore, 
\begin{equation}
\label{eq: Cantor series 3}
x=\sum^{n} _{i=1}{\frac{\varepsilon_i}{q_1q_2\dots q_i}}+\frac{1}{q_1q_2\dots q_n}\sigma^n(x).
\end{equation}

Note that, in the paper \cite{S. Serbenyuk alternating Cantor series 2013}, the notion of the shift operator of an alternating Cantor series is studied in detail.

In \cite{Salem1943}, Salem modeled the function 
$$
s(x)=s\left(\Delta^2 _{\alpha_1\alpha_2...\alpha_n...}\right)=\beta_{\alpha_1}+ \sum^{\infty} _{n=2} {\left(\beta_{\alpha_n}\prod^{n-1} _{i=1}{q_i}\right)}=y=\Delta^{Q_2} _{\alpha_1\alpha_2...\alpha_n...},
$$
where $q_0>0$, $q_1>0$, and $q_0+q_1=1$. This function is a singular function. However,  
generalizations of the Salem function can be non-differentiable functions or do not have the derivative on a certain set.

Let us consider the following generalizations of the Salem function that are described in the paper \cite{S. Serbenyuk systemy rivnyan 2-2} as well. 

\begin{example}[\cite{Symon2015}]
\label{example: 1}
{\rm Let $(q_n)$ is a fixed sequence of positive integers, $q_n>1$, and $(A_n)$ is a sequence of the sets  $\Theta_n = \{0,1,\dots,q_n-1\}$.

Let $x\in [0,1]$ be an arbitrary number represented by a positive Cantor series
\begin{equation*}
%\label{series1}
x=\Delta^Q _{\varepsilon_1\varepsilon_2...\varepsilon_n...}=\sum^{\infty} _{n=1} {\frac{\varepsilon_n}{q_1q_2\dots q_n}}, ~\mbox{where}~\varepsilon_n \in \Theta_n.
\end{equation*}

Let $P=||p_{i,n}||$  be a fixed matrix such that  $p_{i,n}\in (-1,1)$ ($ n=1,2,\dots ,$ and $i=~\overline{0,q_n-1}$), $\sum^{q_n-1} _{i=0} {p_{i,n}}=1$ for an arbitrary $n \in \mathbb N$, and $\prod^{\infty} _{n=1}{p_{i_n,n}}=0$ for any sequence   $(i_n)$.

Suppose that elements of the matrix $P=||p_{i,n,}||$ can be negative numbers as well but   
$$
\beta_{0,n}=0, \beta_{i,n}>0 ~\mbox{for}~ i\ne 0, ~\mbox{and} ~ \max_i {|p_{i,n}|} <1.
$$
Here 
$$
\beta_{\varepsilon_{k},k}=\begin{cases}
0&\text{if $\varepsilon_{k}=0$}\\
\sum^{\varepsilon_{k}-1} _{i=0} {p_{i,k}}&\text{if $\varepsilon_{k}\ne 0$.}
\end{cases}
$$
Then the following statement is true.
\begin{theorem}[\cite{Symon2015}]
Given the matrix $P$  such that for all $n \in \mathbb N$ the following are true:  $p_{\varepsilon_n,n}\cdot p_{\varepsilon_n-1,n}<0$ moreover $q_n \cdot~p_{d_n-1,n}\ge 1$ or  $q_n \cdot p_{q_n-1,n}\le 1$; and the  conditions 
$$
\lim_{n \to \infty} {\prod^{n} _{k=1} {q_k p_{0,k}}}\ne  0, \lim_{n \to \infty} {\prod^{n} _{k=1} {q_k p_{q_k-1,k}}}\ne 0
$$
hold simultaneously.
Then the function  
$$
F(x)=\beta_{\varepsilon_1(x),1}+\sum^{\infty} _{k=2} {\left(\beta_{\varepsilon_k(x),k}\prod^{k-1} _{n=1} {p_{\varepsilon_n(x),n}}\right)}
$$
is non-differentiable on  $[0,1]$.
\end{theorem}
} 
\end{example}
\begin{example}[\cite{Symon2017}]
{\rm
Let
$P=||p_{i,n}||$ be a given matrix such that  $n=1,2, \dots$ and $i=\overline{0,q_n-1}$. For this matrix the following system of properties  holds: 
$$ 
\left\{
\begin{aligned}
\label{eq: tilde Q 1}
1^{\circ}.~~~~~~~~~~~~~~~~~~~~~~~~~~~~~~~~~~~~~~~~~~~~~~~\forall n \in \mathbb N:  p_{i,n}\in (-1,1)\\
2^{\circ}.  ~~~~~~~~~~~~~~~~~~~~~~~~~~~~~~~~~~~~~~~~~~~~~~~~\forall n \in \mathbb N: \sum^{q_n-1}_{i=0} {p_{i,n}}=1\\
3^{\circ}. ~~~~~~~~~~~~~~~~~~~~~~~~~~~~~~~~~~~~~ \forall (i_n), i_n \in  \Theta_{n}: \prod^{\infty} _{n=1} {|p_{i_n,n}|}=0\\
4^{\circ}.~~~~~~~~~~~~~~\forall  i_n \in \Theta_{n}\setminus\{0\}: 1>\beta_{i_n,n}=\sum^{i_n-1} _{i=0} {p_{i,n}}>\beta_{0,n}=0.\\
\end{aligned}
\right.
$$ 

 Let us consider the following function 
$$ 
\tilde{F}(x)=\beta_{\varepsilon_1(x),1}+\sum^{\infty} _{n=2} {\left(\tilde{\beta}_{\varepsilon_n(x),n}\prod^{n-1} _{j=1} {\tilde{p}_{\varepsilon_j(x),j}}\right)},
$$
where
$$
\tilde{\beta}_{\varepsilon_n(x),n}=\begin{cases}
\beta_{\varepsilon_n(x),n}&\text{if $n$ is   odd }\\
\beta_{q_n-1-\varepsilon_n(x),n}&\text{if $n$ is  even,}
\end{cases}
$$
$$
\tilde{p}_{\varepsilon_n(x),n}=\begin{cases}
p_{\varepsilon_n(x),n}&\text{if $n$  is odd }\\
p_{q_n-1-\varepsilon_n(x),n}&\text{if $n$  is   even,}
\end{cases}
$$
$$
\beta_{\varepsilon_{n}(x),n}=\begin{cases}
0&\text{if $\varepsilon_{n}=0$}\\
\sum^{\varepsilon_{n}-1} _{i=0} {p_{i,n}}&\text{if $\varepsilon_{n}\ne 0$.}
\end{cases}
$$
Here $x$ represented by an alternating Cantor series, i.e., 
$$
x=\Delta^{-(q_n)} _{\varepsilon_1\varepsilon_2...\varepsilon_n...}=\sum^{\infty} _{n=1} {\frac{1+\varepsilon_n}{q_1q_2\dots q_n}(-1)^{n+1}},
$$
where $(q_n)$ is a fixed sequence of positive integers, $q_n>1$, and $(\Theta_{n})$ is a sequence of the sets  $\Theta_{n} = \{0,1,\dots,q_n-1\}$, and $\varepsilon_n\in \Theta_{n}$.
\begin{theorem}
Let  $p_{\varepsilon_n,n}\cdot p_{\varepsilon_n-1,n}<0$  for all $n \in \mathbb N$, $\varepsilon_n \in \Theta_{n} \setminus \{0\}$ and conditions 
$$
\lim_{n \to \infty} {\prod^{n} _{k=1} {q_k p_{0,k}}}\ne  0, \lim_{n \to \infty} {\prod^{n} _{k=1} {q_k p_{q_k-1,k}}}\ne 0
$$
hold simultaneously.  Then the function $\tilde{F}$ is  non-differentiable on $[0,1]$. 
\end{theorem}
 }
\end{example}

In the present article, two examples  of certain functions with complicated local structure, are constructed and investigated.

Suppose  that the condition $q_n\le s$ holds for all positive integers $n$.
The first function is following:
$$
f: ~~~ x=\Delta^Q _{\varepsilon_1\varepsilon_2... \varepsilon_n...} ~\longrightarrow~\Delta^q _{\varepsilon_1\varepsilon_2... \varepsilon_n...}=y.
$$
This functon is interesting, since the function described in Example~\ref{example: 1} can be represented by the following way:
$$
F(x)=F_{\xi, Q} \circ f.
$$ 
 Here by ``$\circ$" denote the operation of composition of functions. Also, the function $F_{\xi, Q}$ is a function of the type:
$$
F_{\eta, Q}(y)=\beta_{\varepsilon_1(y),1}+\sum^{\infty} _{k=2} {\left({\beta}_{\varepsilon_k(y),k} \prod^{k-1} _{j=1} {{p}_{\varepsilon_j(y),j}}\right)},
$$
where $y=\Delta^q _{\varepsilon_1\varepsilon_2... \varepsilon_n...}$.

Note that the function $F_{\eta, q}$ is a distribution function of a certain random variable $\eta$ whenever elements $p_{i,n}$ of the matrix $P$ (this matrix described in the last-mentioned examples) are non-negative.
\begin{remark}
Let $\eta$ be a random variable defined by the $q$-ary expansion, i.e.,  
$$
\eta= \frac{\xi_1}{q}+\frac{\xi_2}{q^2}+\frac{\xi_3}{q^3}+\dots+\frac{\xi_{k}}{q^{k}}+\dots \equiv \Delta^{q} _{\xi_1\xi_2...\xi_{k}...},
$$
where the digits   $\xi_k$ $(k=1,2,3, \dots)$ are random and take the values $0,1,\dots , q-1$ with probabilities ${p}_{0,k}, {p}_{1,k}, \dots , {p}_{q-1,k}$. That is, $\xi_k$ are independent and $P\{\xi_k=i_k\}={p}_{i_k,k}$, $i_k \in \Theta=\{0,1, \dots , q-1\}$. 

From the definition of the distribution function and the following expressions for $x=\Delta^q _{\alpha_1\alpha_2...\alpha_k...}$
$$
\{\eta<x\}=\{\xi_1<\alpha_1(x)\}\cup\{\xi_1=\alpha_1(x),\xi_2<\alpha_2(x)\}\cup \ldots 
$$
$$
\ldots \cup\{\xi_1=\alpha_1(x),\xi_2=\alpha_2(x),\dots ,\xi_k<\alpha_k(x)\}\cup \ldots,
$$

$$
P\{\xi_1=\alpha_1(x),\xi_2=\alpha_2(x),\dots ,\xi_k<\alpha_k(x)\}=\beta_{\alpha_k(x),k}\prod^{k-1} _{j=1} {{p}_{\alpha_{j}(x),j}}
$$
we get that the distribution function $F_{\eta, q}$ of the random variable $\eta$ has the
form
$$
F_{\eta, q}(x)=\begin{cases}
0&\text{for $x< 0$}\\
\beta_{\alpha_1(x),1}+\sum^{\infty} _{k=2} {\left[{\beta}_{\alpha_k(x),k} \prod^{k-1} _{j=1} {{p}_{\alpha_j(x),j}}\right]}&\text{for $0 \le x<1$}\\
1&\text{for $x\ge 1$,}
\end{cases}
$$
since the conditions $F_{\eta, q}(0)=0$, $F_{\eta, q}(1)=1$ hold and $F_{\eta, q}$ is a continuous, monotonic, and  non-decreasing function (the most generalized cases of the Salem function were investigated in \cite{S. Serbenyuk function nega-tilde Q-representation}).
\end{remark}

\begin{remark}
In the general case, suppose that $(f_n)$ is a finite or infinite sequence of certain functions (the sequence can contain functions with complicated local structure). Let us consider the corresponding composition of the functions
$$
\ldots \circ f_n \circ \ldots \circ f_2 \circ f_1=f_{c,\infty}
$$
or
$$
 f_n \circ \ldots \circ f_2 \circ f_1=f_{c,n}.
$$
Also, we can take a certain part of the composition, i.e.,
$$
 f_{n_0+t} \circ \ldots \circ f_{n_0+1} \circ f_{n_0}=f_{c,\overline{n_0,{n_0+t}}},
$$
where $n_0$ is a fixed positive integer (a number from the set $\mathbb N$), $t\in \mathbb Z_0=\mathbb N \cup\{0\}$, and $n_0+t \le n$.

One can use such technique for modeling and studying  functions with complicated local structure. Also, one can use  new representations of real numbers (numeral systems) of the type
$$
x^{'}=\Delta^{f_{c,\infty}} _{i_1i_2...i_n}=\ldots \circ f_n \circ \ldots \circ f_2 \circ f_1(x), 
$$
$$
x^{'}=\Delta^{f_{c,n}} _{i_1i_2...i_n}=f_n \circ \ldots \circ f_2 \circ f_1(x)
$$
or
$$
z^{'}=\Delta^{f_{c,\overline{n_0,{n_0+t}}}} _{i_1i_2...i_n}=f_{n_0+t} \circ \ldots \circ f_{n_0+1} \circ f_{n_0}(z).
$$
in fractal theory, applied   mathematics, etc. The next articles of the author of the present article will be devoted to such investigations.
\end{remark}

The second map considered in this article is useful for modeling fractals in space~$\mathbb R^2$. That is, the map
$$
f: x=\Delta^q _{\underbrace{u\ldots u}_{\alpha_1-1}\alpha_1\underbrace{u\ldots u}_{\alpha_2-1}\alpha_2\ldots \underbrace{u\ldots u}_{\alpha_n-1}\alpha_n\ldots} \longrightarrow \Delta^q _{\alpha_1\alpha_2...\alpha_n...},
$$
where $u\in \{0,1, \dots , q-1\}$ is a fixed number,  $\alpha_n \in\{1,2, \dots , q-1\}\setminus\{u\}$, and $3<q$ is a fixed positive integers, models a certain fractal in $\mathbb R^2$. It is easy to see that one can consider such map defined in terms of other representations of real numbers (e.g., the $Q_s$, $Q^{*}$, $Q^{*} _s$, $\tilde Q$, nega-$\tilde Q$-representations and other positive and alternating representations). Really, functions with complicated local structure  defined  in terms of different representations of real numbers, as well as their compositions are useful for modeling fractals (the Moran sets) in $\mathbb R^2$.  Regularities in properties of different sets under the map spawned by functions with complicated local structure and their compositions, are interesting and unknown. The next articles of the author of the present paper will be devoted to such investigations as well.

\section{One function defined in terms of positive Cantor series}

Let us consider the function
$$
f(x)=f\left(\Delta^Q _{\varepsilon_1\varepsilon_2... \varepsilon_n...}\right)=f\left(\sum^{\infty} _{n=1}{\frac{\varepsilon_n}{q_1q_2\cdots q_n}}\right)=\sum^{\infty} _{n=1}{\frac{\varepsilon_n}{q^n}}=\Delta^q _{\varepsilon_1\varepsilon_2... \varepsilon_n...}=y,
$$
where $\varepsilon_n\in\Theta_n$ and the condition $q_n\le q$ holds for all positive integers $n$.

\begin{lemma}[On the well-posedness of the definition of the function]
\label{lemma: 1}
Values of the function $f$ for different representations of Q-rational numbers from $[0,1]$ are:
\begin{itemize}
\item identical whenever for all positive integers $n$ the condition $q_n=q$ holds;
\item different whenever  for all positive integers $n$ the condition $q_n<q$ holds;
\item different for numbers from  no more than a countable subset of Q-rational numbers  whenever there exists a finite or infinite subsequence $n_k$ of positive integers such that $q_{n_k}<q$ for all positive integers values of $k$.
\end{itemize}
\end{lemma}
\begin{proof}
Let $x$ be a Q-rational number. Then there exists a number $n_0$ such that 
$$
x=x_1=\Delta^Q _{\varepsilon_1\varepsilon_2... \varepsilon_{n_0-1}\varepsilon_{n_0}000...}=\Delta^Q _{\varepsilon_1\varepsilon_2... \varepsilon_{n_0-1}[\varepsilon_{n_0}-1][q_{n_0+1}-1][q_{n_0+2}-1][q_{n_0+3}-1]...}=x_2.
$$
Whence,
$$
f(x_1)=\Delta^q _{\varepsilon_1\varepsilon_2... \varepsilon_{n_0-1}\varepsilon_{n_0}000...}, ~~~f(x_2)=\Delta^q _{\varepsilon_1\varepsilon_2... \varepsilon_{n_0-1}[\varepsilon_{n_0}-1][q_{n_0+1}-1][q_{n_0+2}-1][q_{n_0+3}-1]...}
$$
and
$$
f(x_2)-f(x_1)=-\frac{1}{q^{n_0}}+\sum^{\infty} _{n=n_0+1}{\frac{q_n-1}{q^n}}\le 0.
$$
That is, certain Q-rational  points are points of discontinuity of the function. It is easy to see that $f(x_2)-f(x_1)=0$ whenever the condition $q_n=q$ holds for all positive integers $n$.

From unique representation for each Q-irrational number from $[0,1]$ it follows that the  function $f$ is well defined at any Q-irrational  point.
\end{proof}

\begin{remark}
To reach that the function $f$ be well-defined on the set of Q-rational numbers from $[0,1]$, we shall not consider the  representation 
$$
\Delta^Q _{\varepsilon_1\varepsilon_2... \varepsilon_{n-1}[\varepsilon_{n}-1][q_{n+1}-1][q_{n+2}-1][q_{n+3}-1]...}.
$$
\end{remark}

\begin{lemma}
The function $f$ has the following properties:
\begin{enumerate}
\item $D(f)=[0,1]$, where $D(f)$ is the domain of definition of $f$;
\item Let $E(f)$ be the range of values of $f$. Then:
\begin{itemize}
\item $E(f)=[0,1]$ whenever the condition $q_n=q$ holds for all positive integers $n$,
\item $E(f)=[0,1]\setminus C_f$, where $C_f=C_1\cup C_2$,
$$
C_1=\left\{y: y=\Delta^q _{\varepsilon_1\varepsilon_2... \varepsilon_{n}}, \varepsilon_n\notin\{q_n,q_n+1,\dots , q-1\}\text{for all $n$ such that $q_n<q$}\right\}
$$
and
$$
C_2=\left\{y: y=\Delta^q _{\varepsilon_1\varepsilon_2... \varepsilon_{n-1}[\varepsilon_{n}-1][q_{n+1}-1][q_{n+2}-1][q_{n+3}-1]...}\right\};
$$
\end{itemize}
\item $f(x)+f(1-x)=f(1)\le 1$;
\item $f\left(\sigma^k(x)\right)=\sigma^k\left(f(x)\right)$ for any $k\in \mathbb N$.
\end{enumerate}
\end{lemma}
\begin{proof}
\emph{The first property} follows from the definition of $f$. 

\emph{The second property} follows from Lemma \ref{lemma: 1}.

Let us prove \emph{the third property}. Since
$$
1-x=\sum^{\infty} _{n=1}{\frac{q_n-1-\varepsilon_n}{q_1q_2\cdots q_n}},
$$
we have
$$
f(1-x)=\sum^{\infty} _{n=1}{\frac{q_n-1-\varepsilon_n}{q^n}}.
$$
Whence,
$$
f(x)+f(1-x)=\sum^{\infty} _{n=1}{\frac{\varepsilon_n}{q^n}}+\sum^{\infty} _{n=1}{\frac{q_n-1-\varepsilon_n}{q^n}}=\sum^{\infty} _{n=1}{\frac{q_n-1}{q^n}}=f(1)\le 1.
$$
Note that the last inequality is an equality whenever $y=x$, i.e., when the condition $q_n=q$ holds for all positive integers $n$.

Let us prove \emph{the fourth property}. We have
$$
f\left(\sigma^k(x)\right)=f\left(\sum^{\infty} _{j=k+1}{\frac{\varepsilon_{j}}{q_{k+1}q_{k+2}\cdots q_{j}}}\right)=\sum^{\infty} _{j=k+1}{\frac{\varepsilon_{j}}{q^{j-k}}}=\sigma^k\left(\sum^{\infty} _{n=1}{\frac{\varepsilon_{n}}{q^n}}\right)=\sigma^k\left(f(x)\right).
$$
\end{proof}

\begin{lemma}
The function $f$ is continuous at Q-irrational points from $[0,1]$.

The function $f$ is continuous at all Q-rational points from $[0,1]$ if the condition $q_n=q$ holds for all positive integers $n$.

If there exist positive integers $n$ 
 such that $q_n<q$, then points of the type 
$$
\Delta^Q _{\varepsilon_1\varepsilon_2\ldots\varepsilon_{n-1}\varepsilon_n000\ldots}~\text{and} ~~~\Delta^Q _{\varepsilon_1\varepsilon_2\ldots\varepsilon_{n-1}[\varepsilon_n-1][q_{n+1}-1][q_{n+2}-1]\ldots}
$$
are points of discontinuity of the
function.
\end{lemma}
\begin{proof}
 Let $x=\Delta^Q _{\varepsilon_1\varepsilon_2...\varepsilon_n...} \in~[0,1]$ be an arbitrary number.

Let $x_0$ be an Q-irrational number.

Then there exists $n_0=n_0(x)$ such that
$$
\left\{
\begin{array}{rcl}
\varepsilon_m (x)&=&\varepsilon_m (x_0) ~~~\text{for}~~~m=\overline{1,n_0-1}\\
\varepsilon_{n_0} (x)& \ne &\varepsilon_{n_0} (x_0).\\
\end{array}
\right.
$$
From the system, it follows that the conditions $x \to x_0$ and $n_0 \to \infty$ are equivalent and 
$$
\left|f (x) - f (x_0)\right|=\left|\sum^{\infty} _{j=n_0} {\frac{\varepsilon_j (f(x))-\varepsilon_j (f(x_0))}{q^k}}\right|\le \sum^{\infty} _{j=n_0} {\frac{|\varepsilon_j (f(x))-\varepsilon_j (f(x_0))|}{q^k}}\le 
$$
$$
\le \sum^{\infty} _{j=n_0} {\frac{q-1}{q^k}}=\frac{1}{q^{n_0-1}}\to 0 ~\mbox{as} ~n_0 \to \infty.
$$

So, the function $f$  is continuous at Q-irrational points. That is,
$$
\lim_{x\to x_0}{f(x)}=f(x_0).
$$

Let $x_0=\Delta^Q _{\varepsilon_1\varepsilon_2...\varepsilon_n...}$ be a Q-rational number. 

If the condition $q_n<q$ holds for a certain $n\in\mathbb N$, then $q_n\le q-1$ and $q_n-1\le q-2$. That is, 
$$
\varepsilon_n\in\Theta_n=\{0,1,\dots , q_n-1\}\subseteq\{0,1, \dots , q-2\}.
$$
Since

$$
\lim_{x \to x_0-0} {f (x)}=\Delta^q _{\varepsilon_1\varepsilon_2...\varepsilon_{n-1}[\varepsilon_n-1][q_{n+1}-1][q_{n+2}-1]...}
$$
and
$$
\lim_{x \to x_0+0} {f}(x)=\Delta^q _{\varepsilon_1\varepsilon_2...\varepsilon_{n-1}\varepsilon_n000...},
$$
we obtain
$$
\Delta_f=\lim_{x \to x_0+0} {f}(x)-\lim_{x \to x_0-0} {f (x)}=\frac{1}{q^n}-\sum^{\infty} _{j=n+1}{\frac{q_j-1}{q^j}}\ge 0. 
$$

Note that
$$
\Delta_f\ge \frac{1}{q^n}-\sum^{\infty} _{j=n+1}{\frac{q-2}{q^j}}=\frac{1}{(q-1)q^n }
$$
and
$$
\Delta_f\le \frac{1}{q^n}-\sum^{\infty} _{j=n+1}{\frac{1}{q^j}}=\frac{q-2}{(q-1)q^n}.
$$
So, $x_0$ is a point of discontinuity for $q_n<q$ and
$$
\frac{1}{(q-1)q^n }\le \Delta_f\le\frac{q-2}{(q-1)q^n}.
$$
\end{proof}

\begin{lemma}
The function $f$ is strictly increasing.
\end{lemma}
\begin{proof}
Let us have $x_1=\Delta^Q _{\alpha_1\alpha_2...\alpha_n...}$ and $x_2=\Delta^Q _{\varepsilon_1\varepsilon_2...\varepsilon_n...}$ such that $x_1<x_2$. Then there exists $n_0$ such that $\alpha_i=\varepsilon_i$ for $i=\overline{1,n_0-1}$ and $\alpha_{n_0}<\varepsilon_{n_0}$. So,
$$
f(x_2)-f(x_1)=\frac{\varepsilon_{n_0}-\alpha_{n_0}}{q^{n_0}}+\sum^{\infty} _{j=n_0+1}{\frac{\varepsilon_j-\alpha_j}{q^j}}.
$$
Since $\varepsilon_{n_0}>\alpha_{n_0}$ and $q_n\le q$, we have
$$
f(x_2)-f(x_1)>\frac{1}{q^{n_0}}-\sum^{\infty} _{j=n_0+1}{\frac{q_j-1}{q^j}}\ge \frac{1}{q^{n_0}}+\sum^{\infty} _{j=n_0+1}{\frac{1-q }{q^j}}=0.
$$
\end{proof}

\begin{theorem}[On differential properties]
 \  \ \\
\begin{itemize}
\item  If the condition $q_n=q$ holds for all positive integers $n$, then $f^{'} (x_0)=1$;
\item  If for all $n$ the condition $q_n<q$ holds or there exists only a finite numbers of $n$ such that $q_n=q$, then $f$ is a singular function;
\item If   there exists only a finite numbers of $n$ such that $q_n<q$, then  $f$ is non-differentiable;
\item If there exists an infinite subsequence $(n_k)$ of positive integers such that $q_{n_k}<q$, then  $f$ is a singular function.
\end{itemize}
\end{theorem}
\begin{proof}
Suppose $x_0=\Delta^Q _{\varepsilon_1\varepsilon_2...\varepsilon_{m-1}c\varepsilon_{m+1}...}$, where $c$ is a fixed digit from $\{0,1,\dots , q_m-~1\}$, and $(x_m)$ is a sequence of numbers $x_m=\Delta^Q _{\varepsilon_1\varepsilon_2...\varepsilon_{m-1}\varepsilon_m\varepsilon_{m+1}...}$. Then
$$
x_m-x_0=\frac{\varepsilon_m-c}{q_1q_2\cdots q_m}~~~\text{and}~~~f(x_m)-f(x_0)=\frac{\varepsilon_m-c}{q^m}.
$$
Note that the conditions $x_m\to x_0$ and $m\to\infty$ are equivalent. We have
\begin{equation}
\label{eq: 11}
\lim_{m\to\infty}{\frac{f(x_m)-f(x_0)}{x_m-x_0}}=\lim_{m\to\infty}{\frac{q_1q_2\cdots q_m}{q^m}}.
\end{equation}

Let us consider cylinders $\Delta^Q _{c_1c_2...c_n}$. \emph{The change $\mu_{f}$ of the function $f$ on  a cylinder $\Delta^{Q} _{c_1c_2...c_n}$} is called the value   $\mu_{f}\left(\Delta^{Q} _{c_1c_2...c_n}\right)$ defined by the following equality 
\begin{equation*}
\begin{split}
\mu_{f}\left(\Delta^{Q} _{c_1c_2...c_n}\right)&=f\left(\sup \Delta^{Q} _{c_1c_2...c_n}\right)-f\left(\inf \Delta^{Q} _{c_1c_2...c_n}\right)\\ 
&=f\left(\Delta^Q _{\varepsilon_1\varepsilon_2...\varepsilon_n[q_{n+1}-1][q_{n+2}-1]...}\right)-f\left(\Delta^Q _{\varepsilon_1\varepsilon_2...\varepsilon_n000...}\right).
\end{split}
\end{equation*}
So, for $x_0\in\Delta^Q _{c_1c_2...c_n}$, we obtain
\begin{equation}
\label{eq: 12}
f^{'} (x_0)=\lim_{n\to\infty}{\frac{\mu_{f}{\left(\Delta^{Q} _{c_1c_2...c_n}\right)}}{\left|\Delta^Q _{c_1c_2...c_n}\right|}}=\lim_{n\to\infty}{\left(\frac{q_1q_2\cdots q_n }{q^n}\sum^{\infty} _{j=n+1}{\frac{q_j-1}{q^j}}\right)}.
\end{equation}
  Since $2\le q_n\le q$, we have
$$
\frac{1}{q-1}\lim_{n\to\infty}{\left(\frac{q_1q_2\cdots q_n }{q^n}\right)}\le\lim_{n\to\infty}{\left(\frac{q_1q_2\cdots q_n }{q^n}\sum^{\infty} _{j=n+1}{\frac{q_j-1}{q^j}}\right)}\le\lim_{n\to\infty}{\left(\frac{q_1q_2\cdots q_n }{q^n}\right)}.
$$

So,
\begin{itemize}
\item $f^{'} (x_0)=1$ whenever the condition $q_n=q$ holds for all positive integers $n$;
\item $f^{'} (x_0)=0$, i.e, $f$ is a singular function, whenever for all $n$ the condition $q_n<q$ holds or there exists only a finite numbers of $n$ such that $q_n=q$;
\item $f$ is non-differentiable whenever there exists only a finite numbers of $n$ such that $q_n<q$ (since limits \eqref{eq: 11} and \eqref{eq: 12} are different);
\item $f$ is a singular function whenever there exists an infinite subsequence $(n_k)$ of positive integers such that $q_{n_k}<q$.
\end{itemize}
\end{proof}

\begin{theorem}
The Lebesgue integral of the function $f$ can be calculated by the
formula
$$
\int_{[0,1]}{f(x)dx}=\frac{1}{2}\sum^{\infty} _{n=1}{\frac{q_n-1}{q^n}}.
$$
\end{theorem}
\begin{proof}
We have
$$
0\le f(x)\le \sum^{\infty} _{n=1}{\frac{q_n-1}{q^n}}.
$$
Suppose that
$$
T=\{0, \Delta^q _{1000...}, \Delta^q _{2000...}, \dots, \Delta^q _{[q_1-1]000...}, \dots , \Delta^q _{[q_1-1][q_2-1]...[q_n-1]...}\},
$$
$$
E_n=\{x: y_{n-1}\le f(x)<y_n\}=\Delta^Q _{c_1c_2...c_n}, ~~~c_n\in \Theta_n.
$$
We get
$$
\lambda(E_n)=\frac{1}{q_1q_2\cdots q_n},
$$
where $\lambda(\cdot )$ is the Lebesgue measure of a set.

Also, $\overline{y}\in[y_{n-1},y_n)$. Suppose that $\overline{y}=y_{n-1}$. It is easy to see that the conditions $\lambda(E_n)\to 0$ and $n\to \infty$ are equivalent. 

So, 
$$
\int_{[0,1]}{f(x)dx}=\lim_{n\to\infty}{\sum_{n}{\frac{\Delta^q _{c_1c_2...c_n000...}}{q_1q_2\cdots q_n}}}=\lim_{n\to\infty}{\left(\frac{(q_n-1)q_1q_2\cdots q_n}{2\cdot q^n\cdot q_1q_2\cdots q_n}\right)}=\frac{1}{2}\sum^{\infty} _{n=1}{\frac{q_n-1}{q^n}}.
$$
Note that 
$$
\frac{1}{2(q-1)}\le \int_{[0,1]}{f(x)dx}\le \frac{1}{2}
$$
and the integral is equal to $\frac{1}{2}$  whenever $f(x)=x$.
\end{proof}

\section{Fractal in $\mathbb R^2$ defined  in terms of a certain map}

Let us consider the following function
$$
g: x=\Delta^q _{\underbrace{u\ldots u}_{\alpha_1-1}\alpha_1\underbrace{u\ldots u}_{\alpha_2-1}\alpha_2\ldots \underbrace{u\ldots u}_{\alpha_n-1}\alpha_n\ldots} \longrightarrow \Delta^q _{\alpha_1\alpha_2...\alpha_n...},
$$
where $u\in \{0,1, \dots , q-1\}$ is a fixed number,  $\alpha_n \in \Theta=\{1,2, \dots , q-1\}\setminus\{u\}$, and $3<q$ is a fixed positive integers. 
This function can be represented by the following way.
$$
g: x=\frac{u}{s-1}+\sum^{\infty} _{n=1}{\frac{\alpha_n-u}{q^{\alpha_1+\alpha_2+\dots +\alpha_n}}}
 \longrightarrow \sum^{\infty} _{n=1}{\frac{\alpha_n}{q^{n}}}=g(x)=y.
$$
\begin{theorem}
The function $g$ has the following properties:
\begin{enumerate}
\item The domain of definition $D(g)$ of the function $g$ is   an uncountable,   perfect,   and nowhere dense set of zero Lebesgue measure, as well as is a self-similar fractal whose Hausdorff dimension $\alpha_0$ satisfies the following equation 
$$
\sum _{p \ne u, p \in \{1,2,\dots , q-1\}} {\left(\frac{1}{s}\right)^{p \alpha_0}}=1.
$$
\item The range of values of $g$ is a self-similar fractal 
$$
E(g)=\{y: y=\Delta^q _{\alpha_1\alpha_2...\alpha_n...}, \alpha_n\in\Theta\}
$$
whose Hausdorff dimension $\alpha_0$ can be calculated by the formula
$$
\alpha_0(E(g))=\log_q {|\Theta|},
$$
where $|\cdot|$ is the number of elements of a set.
\item The function $g$ on the domain of definition is well defined and is a bijective mapping.
\item On the domain of definition the function $g$ is:
\begin{itemize}
\item decreasing whenever $u\in\{0,1\}$ for all $q>3$;
\item increasing  whenever $u\in \{s-2,s-1\}$ for all $q>3$;
\item  not monotonic  whenever $u\in\{2,3, \dots , s-3\}$ and $q>4$.
\end{itemize}
\item The function $g$ is continuous at any point on the domain.
\item The function $g$ is non-differentiable on the domain.
\item The following relationships are true:
$$
g\left(\sigma^{\alpha_1}(x)\right)=\sigma(g(x)),
$$
$$
g\left(\sigma^{\alpha_1+\alpha_2+\dots +\alpha_n}(x)\right)=\sigma^n(g(x)),
$$
where $\sigma$ is the shift operator.
\item The function does not preserve the Hausdorff dimension.
\end{enumerate}
\end{theorem}
\begin{proof}
For any fixed $u\in \{0,1,\dots , q-1\}$, the domain of definition $D(g)$ of the function $g$ is   an uncountable,   perfect,  and  nowhere dense set of zero Lebesgue measure, as well as is a self-similar fractal whose Hausdorff dimension $\alpha_0$ satisfies the following equation (see \cite{ {S. Serbenyuk 2017  fractals}, {S. Serbenyuk   fractals}})
$$
\sum _{p \ne u, p \in \{1,2,\dots , q-1\}} {\left(\frac{1}{s}\right)^{p \alpha_0}}=1.
$$
This set does not contain q-rational numbers, i.e., numbers of the form
$$
\Delta^q _{\alpha_1\alpha_2...\alpha_n000...}=\Delta^q _{\alpha_1\alpha_2...\alpha_n[q-1][q-1][q-1]...}.
$$
That is, any element of the domain of definition $D(g)$ of the function $g$ has the unique q-representation. Therefore the condition $g(x_1)\ne g(x_2)$ holds for $x_1\ne~x_2$. Note that a value $g(x)\in E(g)$ is assigned to an arbitrary $x\in D(g)$ and and vice versa. 

Let us consider the difference
$$
|g(x)-g(x_0)|=\left|\sum^{\infty} _{n=1}{\frac{\beta_n-\alpha_n}{q^n}}\right|,
$$
where
$x_0=\Delta^q _{\underbrace{u\ldots u}_{\alpha_1-1}\alpha_1\underbrace{u\ldots u}_{\alpha_2-1}\alpha_2\ldots \underbrace{u\ldots u}_{\alpha_n-1}\alpha_n\ldots}$ is a fixed number from $D(g)$ and $x=\Delta^q _{\underbrace{u\ldots u}_{\beta_1-1}\beta_1\underbrace{u\ldots u}_{\beta_2-1}\beta_2\ldots \underbrace{u\ldots u}_{\beta_n-1}\beta_n\ldots}$.
It is easy to see that  the conditions $x\to x_0$ and $\beta_n\to\alpha_n$ are equivalent, $n=1,2, \dots $. So,
$$
\lim_{x\to x_0}{|g(x)-g(x_0)|}=\lim_{\beta_n\to\alpha_n}{\left|\sum^{\infty} _{n=1}{\frac{\beta_n-\alpha_n}{q^n}}\right|}=0.
$$
From the definition of $g$ it follows that the set 
$$
E(g)=\{y: y=\Delta^q _{\alpha_1\alpha_2...\alpha_n...}, \alpha_n\in\Theta\}
$$
is the range of values of $g$. It follows from Theorem~2  in \cite{Serbenyuk 2018 sets} that $E(g)$ is a self-similar fractal whose Hausdorff dimension $\alpha_0$ can be calculated by the formula
$$
\alpha_0(E(g))=\log_q {|\Theta|},
$$
where $|\cdot|$ is the number of elements of a set.

So, \emph{Properties~1--3 and~5} are proved.

Let us prove \emph{Property~4}. Let us have $x_1=\Delta^q _{\underbrace{u\ldots u}_{\alpha_1-1}\alpha_1\underbrace{u\ldots u}_{\alpha_2-1}\alpha_2\ldots \underbrace{u\ldots u}_{\alpha_n-1}\alpha_n\ldots}$  and $x_2=\Delta^q _{\underbrace{u\ldots u}_{\beta_1-1}\beta_1\underbrace{u\ldots u}_{\beta_2-1}\beta_2\ldots \underbrace{u\ldots u}_{\beta_n-1}\beta_n\ldots}$ such that $x_1\ne x_2$. Then there exists $n_0$ such that $\alpha_i=\beta_i$ for $i=\overline{1,n_0-1}$ and $\alpha_{n_0}\ne\beta_{n_0}$. Suppose that $\alpha_{n_0}<\beta_{n_0}$. That is, consider the following numbers
$$
x_1=\Delta^q _{\underbrace{u\ldots u}_{\alpha_1-1}\alpha_1\underbrace{u\ldots u}_{\alpha_2-1}\alpha_2\ldots\underbrace{u\ldots u}_{\alpha_{n_{0}-1}-1}\alpha_{{n_0}-1}  \underbrace{u\ldots u}_{\alpha_{n_0}-1}\alpha_{n_0}\ldots}
$$
and
$$
 x_2=\Delta^q _{\underbrace{u\ldots u}_{\beta_1-1}\beta_1\underbrace{u\ldots u}_{\beta_2-1}\beta_2\ldots\underbrace{u\ldots u}_{\beta_{n_{0}-1}-1}\beta_{{n_0}-1} \underbrace{u\ldots u}_{\beta_{n_0}-1}\beta_{n_0}\ldots}
$$
when $\alpha_{n_0}<\beta_{n_0}$. Note that  sufficiently consider the numbers 
$$
\Delta^q _{\underbrace{u\ldots u}_{\alpha_{n_0}-1}\alpha_{n_0}\ldots}~~~\text{and}~~~\Delta^q _{\underbrace{u\ldots u}_{\beta_{n_0}-1}\beta_{n_0}\ldots}.
$$
Then we obtain the following cases:
\begin{itemize}
\item $g(x_1)<g(x_2)$  for $x_1>x_2$. The last condition is true for the case when $u=0$ or $u=1$. That is, in this case, $g$ is decreasing. 
\item $g(x_1)<g(x_2)$  for $x_1<x_2$. The last condition is true for the case when $u=q-1$ or $u=q-2$.
\item If $g(x_1)<g(x_2)$  for $x_1>x_2$ and $x_1<x_2$. This condition is true for the case when $u\in\{2,3, \dots , q-3\}$ and $q>4$. That is, $g$ is not monotonic.
\end{itemize}

Note that, for $q=4$, $g$ is increasing when $u\in\{2,3\}$ and is decreasing when $u\in\{0,1\}$.

Let us prove \emph{the 6th property}.  Let us consider a sequence $(x_n)$ of numbers $x_n=\Delta^q _{\underbrace{u\ldots u}_{\alpha_1-1}\alpha_1\underbrace{u\ldots u}_{\alpha_2-1}\alpha_2\ldots\underbrace{u\ldots u}_{\alpha_{n-1}-1}\alpha_{{n} -1}  \underbrace{u\ldots u}_{\alpha_{n} -1}\alpha_{n}\underbrace{u\ldots u}_{\alpha_{n+1}-1}\alpha_{n+1}\ldots}$ and a fixed number  $x_0=\Delta^q _{\underbrace{u\ldots u}_{\alpha_1-1}\alpha_1\underbrace{u\ldots u}_{\alpha_2-1}\alpha_2\ldots\underbrace{u\ldots u}_{\alpha_{n_{0}-1}-1}\alpha_{{n}-1}  \underbrace{u\ldots u}_{c-1}c\underbrace{u\ldots u}_{\alpha_{n+1}-1}\alpha_{n+1}\ldots}$, where $c$ is a fixed number. Then
$$
\lim_{x\to x_0}{\frac{g(x)-g(x_0)}{x-x_0}}=\lim_{x\to x_0}{\frac{\frac{\alpha_n-c}{q^n}}{\frac{\alpha_n}{q^{\alpha_1+\alpha_2+\dots +\alpha_{n-1}+\alpha_n}}-\frac{c}{q^{\alpha_1+\alpha_2+\dots +\alpha_{n-1}+c}}}}.
$$
$$
=\lim_{\alpha_n\to c}{\frac{(\alpha_n-c)q^{\alpha_1+\alpha_2+\dots +\alpha_{n-1}+\alpha_n+c}}{q^n (\alpha_nq^c-cq^{\alpha_n})}}=\lim_{\alpha_n\to c}{\frac{q^{\alpha_1+\alpha_2+\dots +\alpha_{n-1}+\alpha_n+c}}{q^{c+n}}}.
$$
So, the function is non-differentiable.

\emph{Property 7}. It is easy to see that 
$$
g\left(\sigma^{\alpha_1}(x)\right)=g(\Delta^q _{\underbrace{u\ldots u}_{\alpha_2-1}\alpha_2\ldots\underbrace{u\ldots u}_{\alpha_{n-1}-1}\alpha_{{n} -1}  \underbrace{u\ldots u}_{\alpha_{n} -1}\alpha_{n}\underbrace{u\ldots u}_{\alpha_{n+1}-1}\alpha_{n+1}\ldots})=\Delta^q _{\alpha_2\alpha_3...\alpha_n...}=\sigma(g(x)),
$$
$$
g\left(\sigma^{\alpha_1+\alpha_2+\dots +\alpha_n}(x)\right)=g(\Delta^q _{\underbrace{u\ldots u}_{\alpha_{n+1}-1}\alpha_{n+1}\ldots\underbrace{u\ldots u}_{\alpha_{n+2}-1}\alpha_{n+2}   \ldots})=\Delta^q _{\alpha_{n+1}\alpha_{n+3}...}=\sigma^n(g(x)).
$$

\emph{Property 8}. It is easy to see that there exists a set $S$ such that $\alpha_0(S)\ne \alpha_0({g(S)})$, where $\alpha_0(\cdot)$ is the Hausdorff dimension of a set.
\end{proof}

\begin{theorem}
The Hausdorff dimension of a graph of the function  $g$ is equal to $1$.
\end{theorem}
\begin{proof}

Suppose that 
$$
X=[0,1]\times[0,1]=\left\{(x,y): x=\sum^{\infty} _{m=1} {\frac{\alpha_m}{q^{m}}}, \alpha_{m} \in \Theta_q=\{0,1,\dots ,q-1\},
y=\sum^{\infty} _{m=1} {\frac{\beta_m}{q^{m}}}, \beta_{m} \in \Theta_q \right\}.
$$
Then the set 
$$
\sqcap_{(\alpha_{1}\beta_{1})(\alpha_{2}\beta_{2})...(\alpha_{m}\beta_{m})}=\Delta^{q} _{\alpha_{1}\alpha_{2}...\alpha_{m}}\times\Delta^{q} _{\beta_{1}\beta_{2}...\beta_{m}}
$$
is a square with a side length of $q^{-m}$. This square is called \emph{a square of rank $m$ with the base $(\alpha_{1}\beta_{1})(\alpha_{2}\beta_{2})\ldots (\alpha_{m}\beta_{m})$}.

If  $E\subset X$, then the number
$$
\alpha^{K}(E)=\inf\{\alpha: \widehat{H}_{\alpha} (E)=0\}=\sup\{\alpha: \widehat{H}_{\alpha} (E)=\infty\},
$$
where
$$
\widehat{H}_{\alpha} (E)=\lim_{\varepsilon \to 0} \left[{\inf_{d\leq \varepsilon} {K(E,d)d^{\alpha}}}\right]
$$
and $K(E,d)$ is the minimum number of squares of diameter $d$ required to cover the set $E$, is called \emph{ the fractal cell entropy dimension of the set E.} It is easy to see that $\alpha^{K}(E)\ge \alpha_0(E)$.

From the definition and properties of the function  $g$ it follows that the graph of the function belongs to $\tau=|\Theta|$ squares from  $q^2$
first-rank squares (here $\tau$ is equal to $(s-1)$ for $u=0$ and $\tau$ is equal to $(s-2)$ for $u\ne 0$):
$$
\sqcap_{(i_1i_1)}=\left[\Delta^q _{\underbrace{u\ldots u}_{i_1-1}i_{1}}, \Delta^q _{i_1}
\right],~i_1 \in \Theta_q.
$$

The graph of  the function $f$ belongs to $\tau^2$ squares from $q^4$ second-rank squares:
$$
\sqcap_{(i_1i_2)(i_1i_2)}=\left[\Delta^q _{\underbrace{u\ldots u}_{i_1-1}i_{1}\underbrace{u\ldots u}_{i_2-1}i_{2}}, \Delta^q _{i_1i_2}
\right],~i_1, i_2 \in \Theta_q.
$$

The graph $\Gamma_{g}$ of the function  $g$ belongs to $\tau^m$ squares of rank $m$ with sides $q^{\alpha_1+\alpha_2+\dots+\alpha_m}$ and $q^{-m}$. Then
$$
\widehat{H}_{\alpha} (\Gamma_g)=\lim_{\overline{m \to \infty}} {\tau^m \left(\sqrt{q^{-2(\alpha_1+\alpha_2+\dots+\alpha_m)}+q^{-2m}}\right)^{\alpha}}.
$$
Since $q^{-m(q-1)} \le q^{-(\alpha_1+\alpha_2+\dots+\alpha_m)}\le q^{-m}$, we get 
$$
\widehat{H}_{\alpha} (\Gamma_g)=\lim_{\overline{m \to \infty}} {\tau^m \left(2\cdot q^{-2m}\right)^{\frac{\alpha}{2}}}=\lim_{\overline{m \to \infty}} {\tau^m \left(2\cdot q^{-2m}\right)^{\frac{\alpha}{2}}}=\lim_{\overline{m \to \infty}}
{\left(2^{\frac{\alpha}{2}}\cdot \tau^m \cdot q^{-m\alpha}\right)}
$$
$$
=\lim_{\overline{m \to \infty}}{\left(2^{\frac{\alpha}{2}}\cdot \left(\frac{\tau}{q^{\alpha}}\right)^m\right)}
$$
for $\alpha_1+\alpha_2+\dots+\alpha_m=m$
and 
$$
\widehat{H}_{\alpha} (\Gamma_g)=\lim_{\overline{m \to \infty}} {\tau^m \left(q^{-2m(q-1)}+q^{-2m}\right)^{\frac{\alpha}{2}}}=\lim_{\overline{m \to \infty}}{\left(\left(\frac{\tau^{\frac{1}{\alpha}}}{q}\right)^{2m}+\left(q^{1-q}\tau^{\frac{1}{\alpha}}\right)^{2m}\right)^{\frac{\alpha}{2}}}
$$
for $\alpha_1+\alpha_2+\dots+\alpha_m=m(q-1)$.

It is obvious that if $\left(\frac{\tau}{q^{\alpha}}\right)^m\to 0$,  $\left(\frac{\tau^{\frac{1}{\alpha}}}{q}\right)^{2m}\to 0$,and $\left(q^{1-q}\tau^{\frac{1}{\alpha}}\right)^{2m}\to 0$
 for $\alpha >1$, and the graph of the function has self-similar properties, then  $\alpha^K (\Gamma_g)=\alpha_0(\Gamma_g)=~1$. 
\end{proof}

\end{document}